\documentclass[reqno, a4paper]{amsart}
\usepackage{amsmath, amssymb, amsthm, enumitem, float}
\usepackage{placeins}

\pdfoutput=1
\usepackage{color}
\def\comment#1{}

\usepackage{txfonts}

\usepackage{tikz}
\usetikzlibrary{external}
\usetikzlibrary{backgrounds,patterns,decorations.pathreplacing,decorations.pathmorphing}
\usepackage{caption}
\usepackage{accents}

\usepackage{hyperref}

\newtheorem{theorem}{Theorem}

\newtheorem{definition}[theorem]{Definition}
\newtheorem{lemma}[theorem]{Lemma}
\newtheorem{proposition}[theorem]{Proposition}

{Important Convention}

\theoremstyle{remark}
\newtheorem{remark}[theorem]{Remark}

 \renewcommand{\phi}{\varphi}

\newcommand{\N}{\mathbb{N}}

\newcommand{\R}{\mathbb{R}}

\DeclareMathOperator{\supp}{supp}
\newcommand{\bes}{\begin{subequations}}
\newcommand{\ees}{\end{subequations}}
\newcommand{\eea}{\end{eqnarray}}

\usepackage{bbm}

\renewcommand{\epsilon}{\varepsilon}

\DeclareMathOperator{\proj}{proj}

\newcommand{\fourIdx}[5]{%
\setbox1=\hbox{\ensuremath{^{#1}}}%
 \setbox2=\hbox{\ensuremath{_{#2}}}%
 \setbox5=\hbox{\ensuremath{#5}}%
 \hspace{\ifnum\wd1>\wd2\wd1\else\wd2\fi}%
 \ensuremath{\copy5^{\hspace{-\wd1}\hspace{-\wd5}#1\hspace{\wd5}#3}%
 _{\hspace{-\wd2}\hspace{-\wd5}#2\hspace{\wd5}#4}%
 }}

\numberwithin{equation}{section}
\numberwithin{theorem}{section}

\renewcommand{\subset}{\subseteq}
\renewcommand{\supset}{\supseteq}

\usepackage{subcaption}
\usepackage{graphicx}

\renewcommand{\mathrm}{}

\makeatletter
\newcommand{\mylabel}[2]{#2\def\@currentlabel{#2}\label{#1}}
\makeatother
\begin{document}

\title{Weak monotone rearrangement on the  line}
\author{J. Backhoff-Veraguas}
\author{M. Beiglb\"ock}
\author{G. Pammer}
\maketitle
\vspace{-30pt}
\begin{abstract}
Weak optimal transport has been recently introduced by Gozlan et al. The original motivation stems from the theory of geometric inequalities; further applications concern numerics of martingale optimal transport and  stability in mathematical finance. 

In this note we provide a complete geometric characterization of the \emph{weak} version of the classical monotone rearrangement between measures on the real line, complementing earlier results of Alfonsi, Corbetta, and Jourdain.  
\end{abstract}

\section{Introduction}

Recently, there has been a growing interest in  weak transport problems as introduced by Gozlan et al \cite{GoRoSaTe17}. 
While the original motivation mainly stems from applications to geometric inequalities (cf.\ the works of Marton \cite{Ma96concentration, Ma96contracting} and Talagrand \cite{Ta95, Ta96}), weak transport problems appear also in a number of further topics, including martingale optimal transport \cite{AlCoJo17, AlBoCh18, BeJu17, BaBeHuKa17}, the causal transport problem \cite{BaBeLiZa16,AcBaZa16}, and  stability in math.\ finance \cite{BaBaBeEd19}.

%
%

\subsection{Framework and main results} \label{sec:main results}
Write $\Pi(\mu,\nu)$ for  the set of couplings between $\mu,\nu\in\mathcal P(\R^d)$. Starting with the seminal article of Gangbo-McCann \cite{GaMc96} 
 problems of the form
\begin{align}\label{ClassicOT}\textstyle 
W_\theta(\mu,\nu) := \inf_{\pi\in\Pi(\mu,\nu)} \int_{\R\times\R} \theta(x-y)\pi(dx,dy),\end{align} where    $\theta\colon\R^d\rightarrow \R$ denotes a convex function 
have received particular attention in optimal transport. The pendant in weak optimal transport consists in 
\begin{align}\label{eq:GJ2}\textstyle 
	V_\theta(\mu,\nu) :=   \inf_{\mu^* \leq_c \nu}  W_\theta(\mu,\mu^*).
	\end{align}
 Here $\leq_c$ denotes the convex order, i.e.\ $\mu \leq_c \nu$ iff $\int \phi\, d\mu \leq \int \phi\, d\nu$ for all convex  $\phi:\R^d \to \R$.  

The problem \eqref{eq:GJ2}, and in particular its one dimensional version,  is investigated in \cite{AlCoJo17, GoRoSaTe17, GoRoSaSh18, Sh16, Sa17, Sh18, FaSh18, GoJu18,  BaBePa18,BaBaBeEd19}. The main purpose of this note is to give a complete geometric characterization of the optimizer $\mu^*$ 
in one dimension.
\begin{definition}\label{def:T properties}
	Fix $\mu,\nu\in\mathcal P_1(\R)$. We call a function $S\colon\R\rightarrow \R$ admissible if it satisfies
	\begin{enumerate}[label=(\roman*)]
		\item \label{property 1}$S$ is  increasing,
		\item \label{property 2}$S$ is  1-Lipschitz,
		\item \label{property 3}$S(\mu)\leq_c \nu$.
	\end{enumerate}
\end{definition}


\begin{theorem}\label{thm:equivalence}
Let $\mu,\nu\in\mathcal P_1(\R)$. There exists an admissible $T$ ($\mu$-a.s.\ unique) 
which is maximal in the sense that $S(\mu) \leq_c T(\mu)$ for every other admissible $S$. 

 If  $\theta\colon\R\rightarrow \R$ is  convex  then  $\mu^*:=T(\mu)$ is an optimizer of   \eqref{eq:GJ2}. 
	If  $\theta\colon\R\rightarrow \R$ is strictly convex and $V_\theta(\mu,\nu)$ is finite, $T(\mu)$ is the unique optimizer of   \eqref{eq:GJ2}. 
\end{theorem}
We call the ($\mu$-a.s.\ unique) map $T$ in Theorem \ref{thm:equivalence}  the \emph{weak monotone rearrangement}.

A particular consequence of Theorem \ref{thm:equivalence} is that the optimizer of  \eqref{eq:GJ2} does not depend on the choice of the convex function $\theta$. We find this fact non-trivial as well as  remarkable and highlight that it is not new: different independent proofs were given by Gozlan et al \cite{GoRoSaTe17}, Alfonsi, Corbetta, Jourdain \cite{AlCoJo17} and Shu \cite{Sh18}. Alfonsi, Corbetta and Jourdain \cite[Example 2.4]{AlCoJo19} notice that this does not pertain in higher dimensions.


The map $T$ can be explicitly characterized in geometric terms  using the notion of irreducibility introduced in \cite{BeJu16}: Measures $\eta, \nu\in P_1(\R)$ are in convex order iff 
\begin{align}\textstyle u_\eta(y) := \int_{\R} |x-y| \eta(dx) \leq \int_{\R} |x-y|\nu(dx) =: u_\nu(y),\end{align}
and, by continuity, the set $U$ where this inequality is strict is open. Hence $U=\bigcup_n I_n$, where $(I_n)$ is an at most countable family of disjoint open intervals; these intervals $I_n$ are called \emph{irreducible} with respect to $(\eta, \nu)$.
\begin{theorem}\label{thm:wmr complementary}
	 The weak monotone rearrangement $T$ of $\mu,\nu\in\mathcal P_1(\R)$ is the unique admissible map which has slope 1 on each interval $T^{-1}(I)$,  where $I$ is irreducible wrt $(T(\mu), \nu)$.
\end{theorem}
Theorem \ref{thm:wmr complementary} represents a necessary and sufficient condition for  the optimality  of the measure $T(\mu^*)$ in \eqref{eq:GJ2}. We note that the `necessary' part was first obtained (using somewhat different phrasing) by Alfonsi, Corbetta, and Jourdain \cite[Proposition 3.12]{AlCoJo17}. We also refer the reader to the semi-explicit representation of $T$ and $T(\mu^*)$ given  in \cite{AlCoJo17}.
\vspace{-50mm}

 
 \begin{figure}[H]
    \centering
    \includegraphics[page=1,width=1\textwidth]
        {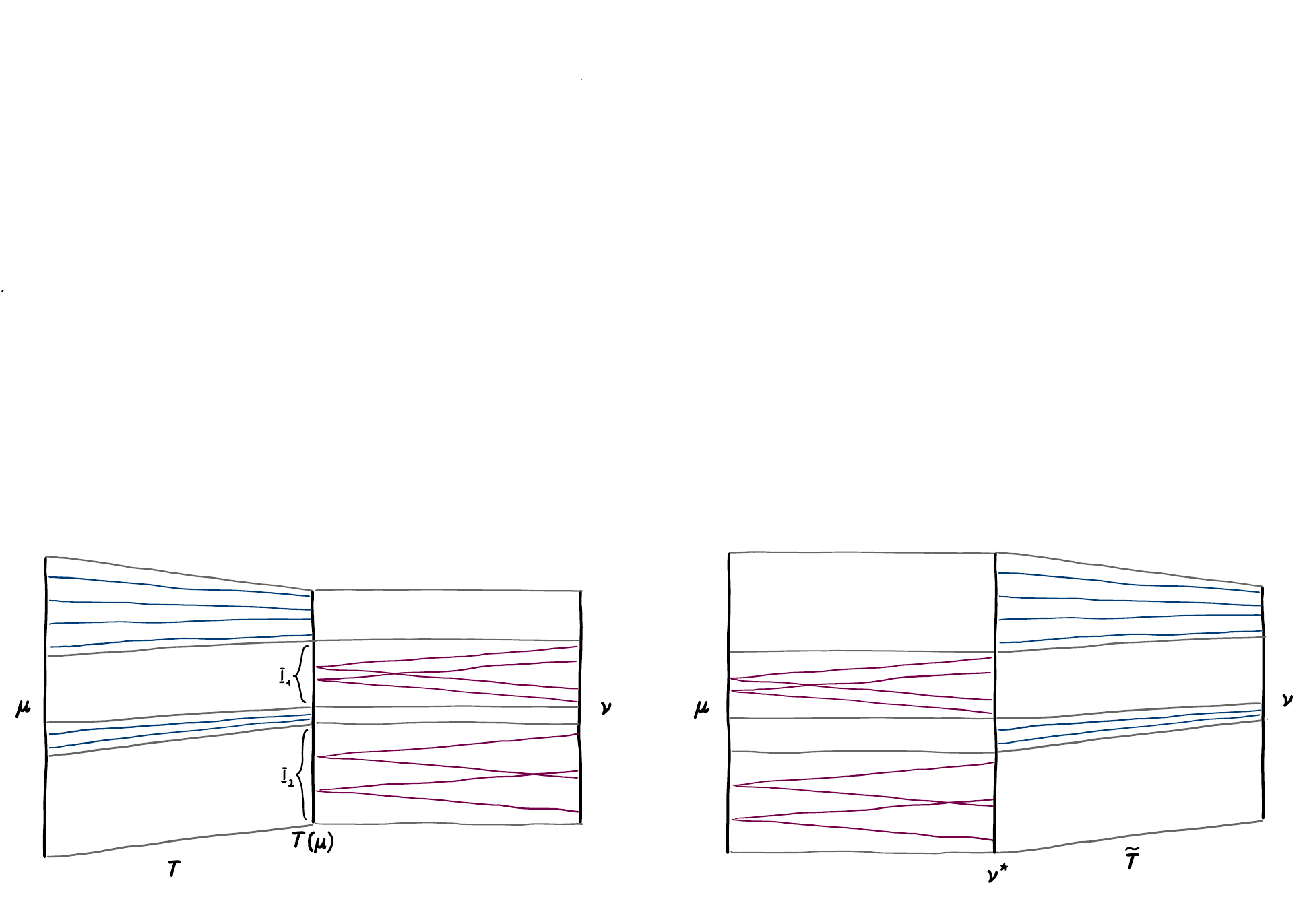} 
  \caption{The solutions to \eqref{eq:GJ2}/\eqref{eq:GJ} and \eqref{eq:GJ3}, respectively. Blue lines depict contractive parts of the map, purple lines depict areas with (non trivial) martingale transport.}
  \label{fig:wmr}
 \end{figure}

\subsection{Connection with martingale transport plans}

Intuitively, the irreducible intervals of $(\eta, \nu)$ are the components where we need to `expand' $\eta$ in order to transform it into $\nu$. In this sense Theorem \ref{thm:wmr complementary} asserts that the mass of $\mu$ can \emph{either} concentrate between  $\mu$ and $T(\mu)$, \emph{or} it can expanded between $T(\mu)$ and $\nu$ (see Figure \ref{fig:wmr}). 


To make this precise, we recall  from \cite{GoRoSaTe17} that \eqref{eq:GJ2} can be reformulated as
\begin{align}\label{eq:GJ}\textstyle 
	V_\theta(\mu,\nu) =  \inf_{\pi\in\Pi(\mu,\nu)} \int_{\R^d} \theta\Big(x-\int_{\R^d} y \pi_x(dy)\Big)\mu(dx),
\end{align}
where $(\pi_x)_{x\in\R^d}$ denotes a regular disintegration of the coupling $\pi$ wrt its first marginal $\mu$.

The set of optimizers of \eqref{eq:GJ} is also straightforward to express in terms of $T$: 
Write $\Pi_M(\eta, \nu)$ for the set of   \emph{martingale couplings} (or martingale transport plans), i.e.\ $\pi\in \Pi(\eta, \nu) $ which satisfy barycenter$(\pi_x)=x$, $\eta$-a.s. 
By 
Strassen's theorem  $\Pi_M(\eta, \nu)$ is nonempty iff $\eta \leq_c \nu$. Using this notation,  $\pi\in \Pi(\mu,\nu)$ is optimal iff there exists a martingale coupling  $\pi^M\in \Pi_M(T(\mu), \nu)$ such that $\pi$ is the concatenation of the transports described by $T$ and $\pi^M$:
\begin{align} \textstyle
\pi(A\times B)= \int_A\mu(dx)\pi_{T(x)}^M(B).
\end{align}
Any  $\pi^M \in \Pi_M(\eta, \nu)$ can be decomposed based on the family of irreducible intervals $(I_n)_n$: denoting $F:=(\cup_n I_n)^c $ by \cite[Appendix A]{BeJu16} we have
\begin{align}\label{DecomposingMM}\textstyle  \pi^M= \sum_n \pi^M_{|I_n \times \bar I_n} + (\mbox{Id},\mbox{Id})(\mu_{|F}).\end{align}
Plainly, \eqref{DecomposingMM} asserts that any martingale transport plan can move mass only within the individual irreducible intervals, whereas particles $x\in F$ have to stay put.


\FloatBarrier

\subsection{A reverse problem}

Alfonsi, Corbetta, and Jourdain \cite{AlCoJo17} proved that the same  value  is obtained when reversing the order of transport  and convex order relaxation in \eqref{eq:GJ2}, i.e.
\begin{align}\label{eq:GJ3} \textstyle 
	V_\theta(\mu,\nu) =   \inf_{\mu \leq_c \nu^*}  W_\theta(\nu^*, \nu);
	\end{align}
	moreover they find (\cite[Proposition 3.12]{AlCoJo17}) a  monotone mapping which is optimal between the optimizer of \eqref{eq:GJ3} and $\nu$ as well as  between $\mu$ and the optimizer of \eqref{eq:GJ2}.
	
As a counterpart to Theorems \ref{thm:equivalence} and \ref{thm:wmr complementary} we establish the following result, strengthening the connection between \eqref{eq:GJ2} and \eqref{eq:GJ3}.
\begin{theorem}\label{thm:equivalence2}
Let $\mu,\nu\in\mathcal P_1(\R)$. 
Then there exists a unique $\leq_c$-smallest measure $\nu^*$, $\mu\leq_c \nu^*$, which can be pushed onto $\nu$ by an increasing 1-Lipschitz mapping. Moreover we have
\begin{enumerate}
\item
 If  $\theta\colon\R\rightarrow \R$ is  convex  then  $\nu^*$ is an optimizer of   \eqref{eq:GJ3}. 
	If  $\theta\colon\R\rightarrow \R$ is strictly convex and $V_\theta(\mu,\nu)$ is finite, $\nu^*$ is the unique optimizer of   \eqref{eq:GJ3}. 
	
\item	There exists a ($\nu^*$-unique) increasing 1-Lipschitz mapping $\tilde T$ which pushes $\nu^*$ onto $\nu$; $\tilde T$ has slope 1 on each interval $I$,  where $I$ is irreducible wrt $(\mu, \nu^*)$.
	
\item	$\tilde T$ is the weak monotone rearrangement between $\mu$ and $\nu$.
\end{enumerate}
\end{theorem}

%
%

%
%
As in the previous section, \eqref{eq:GJ3} could be interpreted as a concatenation of a martingale transport with the weak  monotone rearrangement.

\subsection{An  auxiliary result}
We close this introductory section which an auxiliary result that will be important in the proofs of our results. Since it might be of independent interest we provide the $d$-dimensional version. { We denote the topology induced by the $\rho$-Wasserstein distance on the space of probability measures on $\R^d$ by $\mathcal W_\rho$.}
\begin{theorem}[Stability]\label{thm:stability'}
	Let $1\leq \rho<\infty$, $(\mu^k)_{k\in\N}\in\mathcal P_\rho(\R^d)^\N$, $(\nu^k)_{k\in\N}\in\mathcal P_1(\R^d)^\N$ and $\theta\colon\R^d\rightarrow\R$ convex and such that
for some constant $c>0$ it holds
	\begin{align}\label{eq:theta growth condition}
		\theta(x) \leq  c(1 + |x|^\rho)\quad \forall x\in\R^d. 
	\end{align}	
	 If $\mu^k\rightarrow \mu$ in $\mathcal W_\rho$ and $\nu^k\rightarrow \nu$ in $\mathcal W_1$, then 
	$
		 \lim_k V_\theta(\mu^k,\nu^k) = V_\theta(\mu,\nu)$.
	If additionally $\theta$ is strictly convex, we have that
	\begin{enumerate}
		\item $\arg \min_{\eta\leq_c \nu^k} W_\theta(\mu^k,\nu^k) \rightarrow \arg \min_{\eta\leq_c \nu} W_\theta(\mu,\nu)\quad\text{in }\mathcal W_1$,
		\item the sequence of maps $T^k$, where $T^k(\mu)\leq_c \nu$ and $V_\theta(\mu,\nu^k) = W_\theta(\mu,T^k(\mu))$, converges in $\mu$-probability to $T$, where $T(\mu)\leq_c \nu$ and $V_\theta(\mu,\nu) = W_\theta(\mu,\nu)$.
	\end{enumerate}
\end{theorem}

\section{\textit{C}-Monotonicity implies geometric characterization}

 In this part we prove the following
\begin{theorem}\label{thm:optimality implies fitting}
	Let $\mu,\nu\in\mathcal P_1(\R)$ and $\theta\colon\R\rightarrow \R$ strictly convex. If $V_\theta(\mu,\nu)$ yields a finite value, for any optimizer $\pi\in\Pi(\mu,\nu)$ of \eqref{eq:GJ}, the map
	$$\textstyle T(x) := \int_\R y\pi_x(dy),$$
	is $\mu$-almost surely uniquely defined { and is independent of the specific coupling $\pi$}{, it is admissible in the sense of Definition \ref{def:T properties}, and it has slope $1$ on $T^{-1}(I)$ if $I$ is an irreducible interval for $(T(\mu),\nu)$.} 
\end{theorem}
To prove Theorem \ref{thm:optimality implies fitting}, we need some further properties connected to  irreducibility:

\begin{lemma}\label{lem:MGintersections}
	Suppose  $\{u_\mu < u_\nu\} =: (a,b)$. Then for any $\pi\in\Pi_M(\mu,\nu)$, any regular disintegration $(\pi_x)_{x\in \R}$ wrt $\mu$ and any $c \in (a,b)$  such that
	$\mu((a,c))>0,\, \mu((c,b))>0$,
	there are $x\in (a,c]$, $y\in [c,b)$, $x\neq y$, such that the supports of $\pi_x$ and $\pi_y$ overlap, i.e.
	\begin{align}\label{eq:MTintersections}
	\begin{split}
		\text{\rm int}(\text{\rm co}(\supp(\pi_x))) \cap \text{\rm co}(\supp(\pi_y)) \cup	\text{\rm int}(\text{\rm co}(\supp(\pi_y))) \cap \text{\rm co}(\supp(\pi_x)) \neq \emptyset.
	\end{split}
	\end{align}
\end{lemma}

\begin{proof}
	To show this assertion, we assume the opposite. So there exist $c \in \{u_\mu < u_\nu\}$, $\pi \in \Pi_M(\mu,\nu)$ with fixed disintegration $(\pi_x)_{x\in \R}$ wrt $\mu$ and
	$$\mu((a,c))>0,\quad \mu((c,b))>0,$$
	so that for all $x,y$ with $a<x\leq c\leq y<b$, $x\neq y$, we have 
	\begin{align}\label{eq:notMTintersections}
	\text{\rm int}(\text{\rm co}(\supp(\pi_x))) \cap \text{\rm co}(\supp(\pi_y)) \cup	\text{\rm int}(\text{\rm co}(\supp(\pi_y))) \cap \text{\rm co}(\supp(\pi_x)) = \emptyset.
	\end{align}		
	Since $\pi$ is a martingale coupling, and by \eqref{eq:notMTintersections}, there exists $d\in(a,b)$ with
	\begin{gather}\label{eq:split of concentrations}
	\begin{split}
		\supp(\pi_x)\subset (-\infty,d]\quad\text{for $\mu$-a.e. }x<c,\\
		\supp(\pi_y)\subset [d,\infty) \quad\text{ for $\mu$-a.e. }y>c.
	\end{split}
	\end{gather}
	Write $d_+$ for the largest and $d_-$ for the smallest $d$ such that \eqref{eq:split of concentrations} holds. Note then that $d_-,d_+\in (a,b)$. We have either $\supp(\pi_c) \subset [d_-,d_+]$ or $\mu(\{c\}) =0$, which in any case implies $\mu([c\wedge d_-,c\vee d_+] \setminus \{c\}) = 0$.
	Thus, we infer
	\begin{align*}
	1 &=\pi((-\infty,c) \times (-\infty,d_-) \cup \{c\}\times [d_-,d_+] \cup (c,\infty)\times(d_+,\infty))\\
	&=\pi((-\infty,d_-)\times (-\infty,d_-] \cup \{c\} \times [d_-,d_+] \cup (d_+,\infty)\times [d_+,\infty) ),
	\end{align*}
	and we conclude by contradicting $u_\mu(d_-) < u_\nu(d_-)$ since
	\begin{align*}\textstyle 
	\int_\R |x-d_-|\mu(dx) &=\textstyle  \int_{(a,d_-)}|x-d_-|\mu(dx) + \int_{[d_-,b)} |x-d_-|\mu(dx) \\ &=\textstyle  \int_{(a,d_-)} |y-d_-|\pi_x(dy)\mu(dx) + \int_{[d_-,b)}|y-d_-|\pi_x(dy)\mu(dx) \\ &=\textstyle  \int_\R |x-d_-|\nu(dx). \qquad \qedhere
	\end{align*}
\end{proof}


\begin{lemma}\label{lem:competitor}
	Let $p,q\in \mathcal P_1(\R)$ have overlapping supports (cf.\ \eqref{eq:MTintersections}). Then there exists a continuous map $[0,1]\ni\alpha\mapsto (p_\alpha,q_\alpha)\in\mathcal P(\R)\times\mathcal P(\R)$ such that
	\begin{align}\label{eq:sum property}
		p_\alpha + q_\alpha = p + q\quad \forall \alpha \in [0,1],
	\end{align}
	and such that for some $\beta \in (0,1)$ the functions
	\begin{align}\label{eq:monotone moments}\textstyle 
		 [\beta,1]\ni\alpha \mapsto \int_\R zp_\alpha(dz),\quad  [\beta,1]\ni\alpha \mapsto \int_\R zq_\alpha(dz)
	\end{align}
	are strictly decreasing and increasing, respectively.
\end{lemma}

\begin{proof}
	Let $\alpha \in [0,1]$ and define the inverse distribution functions by
	$$s_\alpha := \inf\{x\in\R\colon F_p(x) \geq \alpha\},\quad t_\alpha := \inf\{x\in\R\colon F_q(x) \geq \alpha\},$$
	where $F_p$ and $F_q$ denote the cumulative distribution functions of $p$ and $q$, respectively. Define two auxiliary measures 
	\begin{align*}
		 \tilde p_\alpha := p|_{(-\infty,s_\alpha)} + (\alpha - F_p(s_\alpha-))\delta_{\{s_\alpha\}},\,\,\,\,
		 \tilde q_\alpha := q|_{(-\infty,t_\alpha)} + (\alpha - F_q(t_\alpha-))\delta_{\{t_\alpha\}}.
	\end{align*}
Defining probability measures $p_\alpha$ and $q_\alpha$ by $p_\alpha := \tilde p_\alpha + \tilde q_{1-\alpha}$ and $q_\alpha := p+q-p_\alpha$,  yields \eqref{eq:sum property} and  continuity of $\alpha \mapsto (p_\alpha,q_\alpha)$. Since $p$ and $q$ satisfy \eqref{eq:MTintersections}, { we find constants $c_1, c_2 \in \supp(p)\cap \supp(q)$ with $c_1 > c_2$} such that
	$$\alpha_1 := p([c_1,+\infty)) > 0\text{ and } \alpha_2 := q((-\infty,c_2])>0.$$
	Let $\alpha_3 < \alpha_1 \wedge \alpha_2$, then for any $1-\alpha_3\leq \alpha< \alpha' \leq 1$ we have
	\begin{align*}\textstyle
	\int_\R zp_{\alpha'}(dz) - \int_\R zp_\alpha(dz) &= \textstyle \int_\R z~ (\tilde p_{\alpha'} - \tilde p_{\alpha})(dz) + \int_\R z~ (\tilde q_{1-\alpha'} - \tilde q_{1-\alpha})(dz)\\	
	&\geq ({\alpha'}- {\alpha})(c_1 - c_2) > 0,
	\end{align*}	
	and conclude that for $\beta := 1 - \alpha_3$ the maps defined in \eqref{eq:monotone moments} are strictly monotone.
\end{proof}
An important tool in the proof of Theorem~\ref{thm:optimality implies fitting} is $C$-monotonicity, a concept which was introduced for the weak optimal transport problem in \cite{BaBeHuKa17, GoJu18, BaBePa18}.
\begin{definition}[$C$-monotonicity]\label{def:C-monotonicity}
	A coupling $\pi \in \Pi(\mu,\nu)$ is $C$-monotone if there exists a measurable set $\Gamma\subseteq X$ with $\mu(\Gamma)=1$, such that for any finite number of points $x_1,\dots,x_N$ in $\Gamma$ and measures $m_1,\dots,m_N$ 
	with $\sum_{i=1}^N m_i = \sum_{i=1}^N \pi_{x_i}$
	\begin{align*}\textstyle 
		\sum_{i=1}^N \theta\left (x_i-\int y\,\pi_{x_i}(dy)\right) \leq \sum_{i=1}^N C\left (x_i-\int y \,m_i(dy)\right).
	\end{align*}
\end{definition}
%

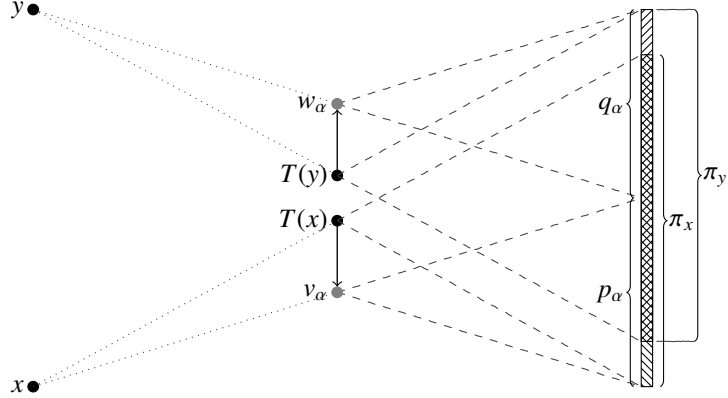
\begin{figure}
\centering
\tikzsetnextfilename{weakmonotonicity}
\begin{tikzpicture}[pencildraw/.style={
    black!75,
    decorate,
    decoration={random steps,segment length=1pt,amplitude=0.1pt}
    }
]
\draw[pencildraw,dotted] (0,0) -- (4,2.2);
\draw[pencildraw,dotted] (0,5) -- (4,2.8);
\draw[black, fill = black] (0,0) circle (2pt);
\draw[black, fill = black] (0,5) circle (2pt);
\draw[black, fill = black] (4,2.2) circle (2pt);
\draw[pencildraw,dashed] (4,2.2) -- (8,4.4);
\draw[pencildraw,dashed] (4,2.2) -- (8,0);
\draw[black, fill = black] (4,2.8) circle (2pt);
\draw[pencildraw,dashed] (4,2.8) -- (8,5);
\draw[pencildraw,dashed] (4,2.8) -- (8,0.6);
\draw (0,0) node[left] {$x$};
\draw (0,5) node[left] {$y$};
\draw[pencildraw,dotted] (0,0) -- (4,1.25);
\draw[pencildraw,dotted] (0,5) -- (4,3.75);
\draw (4,2.2) node[left] {$T(x)$};
\draw (4,2.8) node[left] {$T(y)$};
\draw (4,3.75) node[left] {$w_\alpha$};
\draw[pencildraw,dashed] (4,3.75) -- (8,5);
\draw[pencildraw,dashed] (4,3.75) -- (8,2.5);
\draw (4,1.25) node[left] {$v_\alpha$};
\draw[pencildraw,dashed] (4,1.25) -- (8,2.5);
\draw[pencildraw,dashed] (4,1.25) -- (8,0);
\draw[gray, fill = gray] (4,3.75) circle (2pt);
\draw[gray, fill = gray] (4,1.25) circle (2pt);
\draw[black,->,line width = .5pt] (4,2.8) -- (4,3.675);
\draw[black,->,line width = .5pt] (4,2.2) -- (4,1.325);
\draw[draw=black, pattern=north east lines, pattern color=black](8,0.6) rectangle (8.15,5);
\draw[pencildraw] (8.2,5) -- (8.7,5);
\draw[pencildraw] (8.2,0.6) -- (8.7,0.6);
\draw[decoration={brace}, decorate]  (8.7,5) node {} -- (8.7,0.6);
\draw (8.7,2.775) node[right] {$\pi_y$};
\draw[draw=black, pattern=north west lines, pattern color=black]  (8,0) rectangle (8.15,4.4);
\draw[decoration={brace}, decorate]  (8.25,4.4) node {} -- (8.25,0);
\draw (8.25,2.175) node[right] {$\pi_x$};

\draw[decoration={brace}, decorate]  (7.9,0) node {} -- (7.9,2.5);
\draw[decoration={brace}, decorate]  (7.9,2.5) node {} -- (7.9,5);
\draw (7.9,3.725) node[left] {$q_\alpha$};
\draw (7.9,1.225) node[left] {$p_\alpha$};
\end{tikzpicture}
\caption{Sketch of usage of Lemma~\ref{lem:MGintersections} and Lemma~\ref{lem:competitor} to find contradiction to $C$-monotonicity of an $C$-optimal coupling $\pi$.}
\label{fig:monotonicity and irreducibility}
\end{figure}
\begin{proof}[Proof of Theorem~\ref{thm:optimality implies fitting}]
	Let $\pi^*$ be optimal for the weak optimal transport problem \eqref{eq:GJ}. By the monotonicity principle \cite[Theorem 5.2]{BaBePa18} $\pi^*$ is $C$-monotone, therefore, there exists a set $\Gamma\subset \R$ with $\mu(\Gamma) = 1$ and such that for all
$ x,y\in\Gamma,\,  p_1,p_2\in \mathcal P_1(\R)$ satisfying $ \pi^*_x + \pi^*_y= p_1 + p_2,$
we have
\begin{align}\label{eq:monotonicity}\textstyle 
\theta\Big(x-\int_\R z\pi^*_x (dy)\Big) + \theta\Big(y-\int_\R z\pi^*_y(dz)\Big)\leq \theta\Big(x-\int_\R zp_1(dz)\Big) + \theta\Big(y-\int_\R zp_2(dz)\Big).
\end{align}
As an immediate consequence, we find that the map $T(x) = \int_{\R} y \pi^*_x(dy)$ is $\mu$-almost surely increasing. Letting $x,y\in\Gamma$, for any $\alpha \in [0,1]$ we define 
$$p^\alpha_1 = (1-\alpha) \pi^*_x + \alpha \pi^*_y,\quad p^\alpha_2 = \alpha \pi^*_x + (1-\alpha) \pi^*_y.$$
Plugging $p^\alpha_1$ and $p^\alpha_2$ into \eqref{eq:monotonicity} and computing the righthand-side derivative yields
\begin{align*}\textstyle 
	\Big(\partial_+\theta(x-T(x)) - \partial_+\theta(y-T(y))\Big)(T(x)-T(y))\geq 0,
\end{align*}
which is by strict convexity of $\theta$ equivalent to
\begin{align*}\textstyle 
	\Big(x-T(x)-y+T(y)\Big)(T(x)-T(y))\geq 0.
\end{align*}
Hence, $|x-y| \geq |T(x)-T(y)|$ and $T$ is $\mu$-a.e.\ $1$-Lipschitz.

 Further, we have that $T(\mu)\leq_c\nu$ and $\bar \pi := (T,id)(\pi^*) \in \Pi_M(T(\mu),\nu)$. Without loss of generality, we can assume $\bar \pi_{T(x)} = \pi^*_x,\, \mu\text{-a.e.}$ Let $(I_k)_{k\in\N}$ be the intervals given by the decomposition of $(T(\mu),\nu)$ into irreducible intervals. Assume that there is an interval $I_k$ {so that on $T^{-1}(I_k)$ the map} $T$ does not have $\mu$-a.s.\ slope 1. Then Lemma~\ref{lem:MGintersections} provides two points $\tilde x,\tilde y$ in $T(\Gamma \cap I_k)$ and two corresponding points $x,y\in \Gamma \cap I_k$ such that
$$x < y,\quad T(x) = \tilde x < \tilde y = T(y),\quad x-T(x)> y-T(y),$$
and the overlapping condition \eqref{eq:MTintersections} holds for $\bar \pi_{T(x)} = \pi^*_x$, $\bar \pi_{T(y)} = \pi^*_y$. Lemma~\ref{lem:competitor} allows us to define measures $p_\alpha$ and $q_\alpha$ on $\R$ such that
$$\pi^*_x + \pi^*_y = \bar \pi_{T(x)} + \bar \pi_{T(y)} = p_\alpha + q_\alpha.$$
For a graphical depiction compare with Figure~\ref{fig:monotonicity and irreducibility}. Hence,
$$\textstyle T(x) > v_\alpha := \int_\R z p_\alpha(dz),\quad T(y) < w_\alpha := \int_\R z q_\alpha(dz),$$
are strictly monotone, continuous maps on $[\beta,1]$. Therefore, we find $\alpha \in (\beta,1)$ with
\begin{align*}
	 x - v_1 = x-T(x) < x-v_\alpha \leq y-w_\alpha<y-T(y) = y - w_1.
\end{align*}
By strict convexity of $\theta$ we find
$$ \frac{\theta(y-w_\alpha)-\theta(x-T(x))}{y-w_\alpha-x+T(x)}< \frac{\theta(y-T(y))-\theta(x-\nu_\alpha)}{y-T(y)-x+\nu_\alpha}, $$
which then yields a contradiction to $C$-monotonicity:
\begin{align*}\textstyle 
	\theta\Big(x - \int_\R z\pi^*_x(dz)\Big) + \theta\Big(y - \int_\R z \pi^*_y(dz)\Big) &=\textstyle  \theta(x-T(x)) + \theta(y-T(y)) \\
	&>\textstyle  \theta(x-v_\alpha) + \theta(y-w_\alpha)\\ &=\textstyle  \theta\Big(x - \int_\R zp_\alpha(dz)\Big) + \theta\Big(y-\int_\R zq_\alpha(dz)\Big).\qedhere
\end{align*}
\end{proof}

\section{Sufficiency of the geometric characterization}

Naturally the question arises whether any map $T$ satisfying the properties in Theorem~\ref{thm:optimality implies fitting} must be optimal. The aim of this section is to establish this:

\begin{theorem}\label{thm:fitting implies optimality}
	Let $\mu,\nu\in\mathcal P_1(\R)$. Then any coupling $\pi\in\Pi(\mu,\nu)$ for which $T(x) := \int_{\R} y \pi_x(dy)$  is admissible (in the sense of Definition~\ref{def:T properties}) { with slope $1$ on each interval $T^{-1}(I)$, where $I$ is irreducible wrt $(T(\mu),\nu)$,} is optimal for \eqref{eq:GJ}, i.e.,
\begin{align*}\textstyle
	V_\theta(\mu,\nu) = \int_{\R} \theta(x- T(x)) \mu(dx).
\end{align*}
\end{theorem}

The proof is based on dual optimizers and their explicit representation. As long as $T$ is strictly increasing, \cite[Theorem 2.1]{Sh16} provides dual optimizers to $V_\theta(\mu,T(\mu))$. Investigating dual optimizers further, we are able to show here
$V_\theta(\mu,T(\mu)) = V_\theta(\mu,\nu)$. First, Lemma~\ref{lem:aux map} helps us to carefully approximate the increasing map $T$ with strictly increasing maps $T_\epsilon$.

\begin{lemma}\label{lem:aux map}
	Let $T\colon \R \rightarrow \R$ be an increasing map, with $T-id$ decreasing. Then, for any $\epsilon>0$ there is a strictly increasing map $T_\epsilon\colon\R\rightarrow\R$, with $T_\epsilon-id$ decreasing, such that
	\begin{align*}
		|T-T_\epsilon|_\infty\leq \epsilon,
	\end{align*}
	and 
	 $T$ is affine with slope $1$ on an interval $I$ if and only if $T_\epsilon$ is affine with slope $1$ on $I$.
\end{lemma}

\begin{proof}
	Since $T$ is increasing we know that the pre-image of any point under $T$ corresponds to an interval. Therefore, we can find at most countable many, disjoint intervals $(I_k)_{k\in\N}$ of finite length, where $T(I_k)$ is a singleton and $T$ is strictly increasing on the complement, i.e., on $\bigcap_k I_k^c$. For any $\epsilon>0$, we define $g_\epsilon\colon\R\rightarrow\R$
	\begin{align}\label{eq:geps}
		g_\epsilon(x) := \begin{cases} \frac{\epsilon}{\lambda(I_k)2^k} & \exists k\in\N\colon x \in I_k,\\ 0 &\text{else.} \end{cases}
	\end{align}
Then the map $T_\epsilon(x) := T(x) + \int_{-\infty}^x [g_\epsilon(y) \wedge 1]\,dy$ satisfies the desired properties. 
\end{proof}

Let $T\colon\R\rightarrow \R$ be an increasing  1-Lipschitz   function. Then $T$ induces a unique decomposition of $\R$  into at most countably many maximal, closed, disjoint intervals  $(I_k)_k$ and a ($G_\delta$-set) $G$ such that for all $k\in\N$ the map $T|_{I_k}$ is affine with slope 1 and $T|_G$ is properly contractive, i.e., for any two points $x,y \in G$ we have $|T(x) - T(y)| < |x-y|$. Below we call the intervals  $(I_k)_k$ \emph{irreducible} wrt $T$.

\begin{proof}[Proof of Theorem~\ref{thm:fitting implies optimality}]
The convex function $\theta$ can be approximated by a pointwise-increasing sequence of {Lipschitz} convex functions $(\theta_n)_{n\in\N}$, e.g.
$$\theta_n(x) := \theta(x)\,1_{[-n,n]}(x)+[(x-n)\partial_+\theta(n)]\, 1_{(n,+\infty)}+[\theta(-n) + (x+n)\partial_-\theta(-n)]\,1_{(-\infty,-n)}.$$
%
By monotone convergence we find\footnote{Recall that by  \cite{GoRoSaSh18} the optimizer of the  weak transport problem does not depend on the convex cost.}
$ 	\sup_n V_{\theta_n}(\mu,\nu) = V_\theta(\mu,\nu)$. 
Indeed, if $\pi^n$ optimizes $V_{\theta_n}(\mu,\nu)$ and assuming wlog\ that $\pi_n\to\pi$, then 
\begin{align*}\textstyle 
\lim_n \int \theta_n\left(x-\int y\pi_x^n(dy)\right)\mu(dx)&\geq \textstyle \lim_n \int \theta_m\left(x-\int y\pi_x^n(dy)\right)\mu(dx)\\ &\geq \textstyle  \int \theta_m\left(x-\int y\pi_x(dy)\right)\mu(dx),
 \end{align*}
by \cite[Proposition 2.8]{BaBePa18}. Thus the claim follows by taking the supremum in $m$.

From this, it suffices to consider the case when $\theta$ is Lipschitz continuous.
By Lemma~\ref{lem:aux map} we find for any $\epsilon>0$ a strictly increasing map $T_\epsilon$, such that $T_\epsilon - id$ is decreasing, the decompositions of $T$ and $T_\epsilon$ match, and $|T_\epsilon - T|_\infty \leq \epsilon.$ Then \cite[Theorem 2.1]{Sh16} provides a convex, Lipschitz continuous function $f_\epsilon\colon\R\rightarrow \R$ such that for all $x\in\R$
$$\textstyle  Rf_\epsilon(x) := \inf_{y \in \R} f_\epsilon(y) + \theta(x-y) = f_\epsilon(T_\epsilon(x)) + \theta(x-T_\epsilon(x)),$$
which is even affine on the parts where $T_\epsilon$ is affine. Write $S_\epsilon(x) = \int_{-\infty}^x [g_\epsilon(z)\wedge 1]dz$, where $g_\epsilon$ is defined as in \eqref{eq:geps} so $T_\epsilon=T+S_\epsilon$.
In the following we will show that $f_\epsilon$ is a dual optimizer of the coupling $\pi^\epsilon$ defined as the push-forward measure of $\pi$ by the function
$$(x,y) \mapsto (x,y+S_\epsilon(x)).$$
First, we compute the barycenters of $\pi^\epsilon_x$:
\begin{align*}\textstyle 
	\int_\R y \pi^\epsilon_x(dy) = \int_\R y + S_\epsilon(x) \pi_x(dy)= T_\epsilon(x),
\end{align*}
and $T_\epsilon(\mu) \leq_c \proj_2(\pi^\epsilon) =: \nu_\epsilon$. 
Given the sets $(I_k)_{k\in\N}$ and $F$ from the decomposition of $(T(\mu),\nu)$ into irreducible intervals, we find the sets $(I_k^\epsilon)_{k\in\N}$ and $F_\epsilon$ from the decomposition of $(T_\epsilon(\mu),\nu_\epsilon)$ into irreducible intervals by setting
$$\textstyle I^\epsilon_k := I_k + S_\epsilon(x)\quad \text{for any }y\in I_k \text{ and } x \in T^{-1}(y),\quad F_\epsilon := \bigcap_k (I^\epsilon_k)^c.$$
In view of the structure of martingale couplings, see \cite[Theorem A.4]{BeJu16}, we find that for $\mu$-a.e.\ $x\in T_\epsilon^{-1}(I^\epsilon_k)$ we have $\supp(\pi_x^\epsilon) \subset I^\epsilon_k$ and $\pi_x^\epsilon = \delta_{T(x) + S_\epsilon(x)}$ $\mu$-a.e.\ on $F_\epsilon$.
{Since the decompositions of $T$ and $(T(\mu),\nu)$ are complementary, we infer the same for the decompositions of the map $T_\epsilon$ and the pair $(T_\epsilon(\mu),\nu_\epsilon)$.} The next computation establishes duality of the pair $(\pi^\epsilon,f_\epsilon)$, where we use affinity of $f_\epsilon$ on the irreducible components of $(T_\epsilon(\mu),\nu_\epsilon)$:
\begin{align*}
\textstyle \int_\R \theta\left(x - \int_\R y \pi^\epsilon(dy)\right) \mu(dx) &= \textstyle  \int_\R \theta\left(x-T_\epsilon(x)\right)\mu(dx)\,= \int_\R Rf_\epsilon(x) - f_\epsilon(T_\epsilon(x))\mu(dx)\\
&= \textstyle \int_\R Rf_\epsilon(x)\mu(dx) - \int_\R f_\epsilon\left(\int_\R y \pi_x^\epsilon(dy)\right)\mu(dx)\\
&= \textstyle \int_\R Rf_\epsilon(x)\mu(dx) - \int_\R f_\epsilon(y) \nu_\epsilon(dy).
\end{align*}
This easily proves that $\pi^\epsilon$ is optimal for the optimal weak transport problem
$ V_\theta(\mu,\nu_\epsilon)$.
Drawing the limit for $\epsilon\searrow 0$, we observe
$$\textstyle V_\theta(\mu,\nu_\epsilon) = \int_{\R} \theta(x-T_\epsilon(x))\mu(dx) \rightarrow \int_{\R}\theta(x-T(x))\mu(dx).$$
{ As $\theta$ is Lipschitz, we can apply stability Theorem~\ref{thm:stability'} and obtain optimality of $\pi$.}
\end{proof}

\section{Geometry of the weak monotone rearrangement}

{We can summarize Theorems \ref{thm:optimality implies fitting} and \ref{thm:fitting implies optimality} as follows: There exists an admissible map $T$ with slope $1$ on $T^{-1}(I)$ whenever $I$ is an irreducible interval wrt $(T(\mu),\nu)$, such that $\pi$ is optimal for \eqref{eq:GJ} iff $T(x)=\int y\pi_x(dy)$ ($\mu$-a.s.).} We now show that this map is the weak monotone rearrangement and is therefore the maximum in convex order of the set\footnote{We abuse terminology here, meaning maximum of $\{S(\mu):S\in M(\mu,\nu)\}$.}
$$M(\mu,\nu) := \left\{S\colon\R\rightarrow \R\colon S \text{ is increasing and 1-Lipschitz}, S(\mu)\leq_c \nu \right\}.$$
Heuristically speaking, if the maximum in convex order of the set $M(\mu,\nu)$ is again given by an increasing, 1-Lipschitz map, then this map is as close as possible to a shifted identity. In turn, this is favourable when trying to find the minimum in convex order of
$$\left\{(id-S)(\mu)\colon S\in M(\mu,\nu)\right\},$$
which gives reason to why there should exist a single optimizer to \eqref{eq:GJ} for all convex $\theta$.

As preparation to establishing Theorem~\ref{thm:wmr complementary}, we prove Lemma~\ref{lem:maximum of maps} and Lemma~\ref{lem:maximum minimizes id-map}.

\begin{lemma}\label{lem:maximum of maps}
	Let $\mu\in\mathcal P_1(\R)$, $T,S\colon\R\rightarrow\R$ be increasing maps with
	$$\textstyle \int_\R T(x)\mu(dx) = \int_\R S(x)\mu(dx),$$
	then the maximum (wrt the convex order) of $T(\mu)$ and $S(\mu)$, which is uniquely determined by its potential functions, is again given by an increasing map. If in addition, the maps are $L$-Lipschitz with $L>0$, then the maximum is also given by an $L$-Lipschitz map.
\end{lemma}

\begin{proof}
	The maximum of $T(\mu)$ and $S(\mu)$ wrt.\ the convex order is uniquely determined by the maximum of its potential functions, i.e.\ $u_{S(\mu)}\vee u_{T(\mu)} = u_{S(\mu) \vee T(\mu)}$.
The right-hand side derivative of the potential function can be expressed by the cumulative distribution function, namely
	$\partial_+ u_\mu(x) = 2F_\mu(x) - 1$. 
	{	By continuity of the potential functions, we find a partition of $\R$ into at most countably many disjoint intervals $(I_k)_{k\in\N}$, where $u_{T(\mu)} = u_{S(\mu)}$ on $\partial I_k$, and restricted onto $I_k$ one of the following holds true:
	\begin{enumerate}[label = (\alph*)]
		\item \label{item 1} $u_{S(\mu)}|_{I_k} \geq u_{T(\mu)}|_{I_k}$,
		\item $u_{S(\mu)}|_{I_k} \leq u_{T(\mu)}|_{I_k}$.
	\end{enumerate}
	Suppose wlog \ref{item 1} holds, then
	\begin{align}\label{eq:cdf of max}
		\begin{split}
		F_{S(\mu)\vee T(\mu)}(x) = F_{T(\mu)}(x)\quad x \in (l_k,r_k).
	\end{split}
	\end{align}
	By monotonicity, we can define $T^*$ on $\tilde I_k := (T^{-1}(l_k) \vee S^{-1}(l_k), T^{-1}(r_k) \vee S^{-1}(r_k)]$\footnote{We use the convention that the maximum of the empty set equals $-\infty$.} by	
	\begin{align}\label{eq:def T^*}
	\begin{split}
		T^*(x) = \begin{cases} T(x) & x \in T^{-1}(I^k) \cap \tilde I_k,\\ r_k & \text{else.} \end{cases}
		\end{split}
	\end{align}
	Hence, $T^*$ is an increasing map, $F_{T^*(\mu)} = F_{S(\mu)\vee T(\mu)}$ and $S(\mu)\vee T(\mu)$ is given by the map $T^*$.} If $T$ and $S$ are in addition $L$-Lipschitz, it follows by construction that $T^*$ is $L$-Lipschitz.
	\end{proof}

\begin{lemma}\label{lem:maximum minimizes id-map}
Let $\eta_1\leq_c\eta_2$ and $T_2(\eta_2)\leq_c T_1(\eta_1)$, where $T_1,T_2$ are increasing and $1$-Lipschitz. Then $(id-T_1)(\eta_1)\leq_c(id-T_2)(\eta_2)$. 

In particular, if $T$ and $S$ are increasing $1$-Lipschitz maps s.t.\
	$ \int_\R T(x) \mu(dx) = \int_\R S(x)\mu(dx)$, and we denote by $R$ the increasing $1$-Lipschitz map with
	$R(\mu) = S(\mu)\vee T(\mu)$,
	which exists by Lemma~\ref{lem:maximum of maps}, then
	$(id-R)(\mu) \leq_c (id-S)(\mu)\wedge (id-T)(\mu)$.
\end{lemma}

\begin{proof}
By approximation, it suffices to settle the case when $\eta_1,\eta_2$ are uniform measures on $n\in\mathbb N$ atoms. Let $x^i_1\leq x^i_2 \leq \dots\leq x^i_n$ denote the atoms of $\eta_i$. Then the vector $z^i:= (T_i(x^i_1),\dots,T_i(x^i_n))$ is ordered in an increasing way. What is more, the vector $y^i:=(x^i_1-z^i_1,\dots,x^i_n-z^i_n)$ is likewise ordered increasingly, since $id-T_i$ is an increasing map. By e.g.\ \cite[Proposition 2.6]{GoRoSaSh18} we know that 
$$\textstyle \forall k\leq n:\,\, \sum_{\ell\leq k}x^2_\ell\leq \sum_{\ell\leq k}x^1_\ell\,\,\,\, ,\,\,\,\,  \sum_{\ell\leq k}z^1_\ell\leq \sum_{\ell\leq k}z^2_\ell\,. $$
But then also $\sum_{\ell\leq k} x^2_\ell-z^2_\ell \leq \sum_{\ell\leq k} x^1_\ell-z^1_\ell $, so again by \cite[Proposition 2.6]{GoRoSaSh18} we conclude $(id-T_1)(\eta_1)\leq_c(id-T_2)(\eta_2)$. The second statement easily follows from the first one.
\end{proof}

%

\begin{proof}[Proof of Theorem~\ref{thm:wmr complementary}]
{ Existence of an admissible map $T$ which has slope 1 on each interval $T^{-1}(I)$, where $I$ is irreducible wrt $(T(\mu),\nu)$, was already shown in Theorem~\ref{thm:optimality implies fitting}.} Therefore, it remains to show that the map is maximal. Denote by $T$ the map given by Theorem~\ref{thm:optimality implies fitting} associated with an optimizer to \eqref{eq:GJ} and some strictly convex $\theta\colon\R\rightarrow\R$. Let $S$ be an arbitrary map in $M(\mu,\nu)$. Then Lemma~\ref{lem:maximum minimizes id-map} states that
$$(id-R)(\mu)\leq_c (id-T)(\mu),$$
where $R$ is defined as the increasing, 1-Lipschitz map such that
$R(\mu) = S(\mu)\vee T(\mu)$.
Additionally to existence, strict convexity of $\theta$ ensures $\mu$-almost sure uniqueness of $T$ in the sense that for any optimal coupling $\pi$ we have
$\textstyle \int_\R y \pi_x(dy) = T(x)\quad \mu\text{-a.s.}$
{ Thus, $R(\mu) = T(\mu)$ and $T = R$ $\mu$-almost surely.}
%
%
\end{proof}

{
\begin{proof}[Proof of Theorem~\ref{thm:equivalence}]
	This is a direct consequence of Theorem~\ref{thm:optimality implies fitting}, which provides existence of a map with the desired geometric properties, and Theorem~\ref{thm:wmr complementary}, which provides the equivalence between the geometric properties and maximality.
\end{proof}
}

\section{On the reverse problem of Alfonsi, Corbetta, and Jourdain}

We aim to prove Theorem \ref{thm:equivalence2} pertaining the \emph{reverse} problem \eqref{eq:GJ3}.

\begin{lemma}\label{lem:minimum of maps}
	Let $\eta_1,\eta_2\in\mathcal P_1(\R)$, $T_1,T_2\colon\R\rightarrow\R$ be increasing maps with
	$$T_1(\eta_1) = T_2(\eta_2) =: \nu.$$
	Denote the minimum in convex order of $\eta_1$ and $\eta_2$ by $\eta$. Then there exists an increasing map $T^*$ such that $T^*(\eta) = \nu$. If in addition, the maps are $L$-Lipschitz with $L>0$, then the same holds true for $T^*$.
\end{lemma}

\begin{proof}
	Suppose there exist increasing maps $T_i$ and measures $\eta_i$, $i=1,2$, such that $T_1(\eta_1) = \nu = T(\eta_2)$. Then the potential function of the minimum $\eta$ of $\eta_1$ and $\eta_2$ wrt the convex order is given by the convex hull of $u_{\eta_1}$ and $u_{\eta_2}$. The potential function $u_\eta$ completely specifies the cumulative distribution function through $\partial_+ u_\eta = 2F_\eta - 1$. Thus, we can find a partition of $\R$ into countably many, disjoint intervals $I_k = [a_k,b_k)\cap \R$. For each $k\in\N$, we have $i\neq j \in \{1,2\}$ with $u_\eta(a_k) = u_{\eta_i}(a_k)$, $u_\eta(b_k) = u_{\eta_j}(b_k)$ such that one of the following holds
\begin{enumerate}[label=(\alph*)]
	\item\label{it:decomp map1}$u_\eta(x) = u_{\eta_i}(x)$ on $I_k$,
	\item\label{it:decomp map2}$F_{\eta}(a_k) = F_{\eta}(b_k-) = F_{\eta_j}(b_k-)$ and $u_\eta(x) < u_{\eta_1}(x) \wedge u_{\eta_2}(x)$ on $(a_k,b_k)$.
\end{enumerate}
According to this decomposition, we can define an increasing map $T^*$ via
\begin{align*}\textstyle
	T^*(x) = \begin{cases} T_i(x) & x \in I_k,~\text{\ref{it:decomp map1} holds,} \\ T_i(a_k) & x \in I_k,~\text{\ref{it:decomp map2} holds.}
	\end{cases}
\end{align*}
Note that $T^*$ is $L$-Lipschitz if $T_i$, $i=1,2$, are $L$-Lipschitz. Let $y\in\R$, due to the continuity of the maps $T_1$ and $T_2$, we can find  points $x_1,x_2\in \R$ with $p = F_{\nu}(y)$, $x_1 = F^{-1}_{\eta_1}(p)$ $x_2 = F^{-1}_{\eta_2}(p)$. Assume that $i = 1, j = 2$ with $x_1 \in I_k$. If \ref{it:decomp map1} holds, then we have $F_\eta(x_1) = F_{\eta_1}(x_1)$. Now presume that \ref{it:decomp map2} holds, then $F_{\eta_2}(b_k-) = F_{\eta}(a_k) \leq F_{\eta_1}(x_1)$. Then
\begin{align*}
	&F_{\eta_2}(b_k) = p \implies x_2 = b_k \text{ and }F_{\eta}(b_k) = p,\\
	&F_{\eta_2}(b_k) < p \implies \exists l \in \N \colon x_2 \in I_l\text{ s.t. $i=2$ and \ref{it:decomp map1} holds}.
\end{align*}
Hence, by monotonicity of the map $T^*$ we conclude $F_{T^*(\eta)} = F_\nu$.
\end{proof}

\begin{proof}[Proof of Theorem \ref{thm:equivalence2}]

	Define $\nu^*$ via the weak monotone rearrangement $T$ between $\mu$ and $\nu$, as $\mu$ on the contraction parts of $T$ and accordingly shifted $\nu$ on the affine (irreducible) intervals such that $\tilde T(\nu^*) = \nu$. Let $\eta\geq_c \mu$. Then for any strictly convex $\theta\colon\R\rightarrow\R$ and coupling $\pi^2 \in \Pi(\eta,\nu)$ we have $	\int_{\R\times\R} \theta(y-z)\pi^2(dy,dz) \geq \int_\R \theta(y - \int_\R z \pi^2_y(dz)) \eta(dy),
	$ 	with equality iff $\pi^2$ is actually given by a map. Hence, if the optimizer $\pi^2$ of $W_\theta(\eta,\nu)$ is not given by a map and $\pi^1 \in \Pi_M(\mu,\eta)$, we have
	$$\textstyle W_\theta(\eta,\nu) >  \int_\R \theta(x - \int_\R \int_\R z \pi^2_y(dz) \pi^1_x(dy)) \mu(dx) \geq V_\theta(\mu,\nu).$$
	Thus, by the structure of the weak monotone rearrangement, we deduce optimality of {$\nu^*$} for Problem \eqref{eq:GJ3}.
	To show uniqueness of \eqref{eq:GJ3}, assume that $\eta$ attains the minimum of \eqref{eq:GJ3} and the optimizer of $W_\theta(\eta,\nu)$ is given by the map $R$. For any martingale coupling $\pi^1 \in \Pi_M(\mu,\eta)$ we define a map by
	$$\textstyle L(x) := \int_\R R(y)\pi^1_x(dy).$$
	Then by optimality $L(\mu) = T(\mu)$ and, in particular,
	$$\textstyle \int_\R \theta(y - R(y)) \eta(dy) = \int_\R \theta(x-L(x)) \mu(dx) = \int_\R \theta(x-T(x))\mu(dx),$$
	which shows $L = T$ $\mu$-almost surely. By strict convexity, we have
	$$\textstyle y - R(y) = x - L(x) = x - T(x)\quad \pi^1\text{-a.s.}$$
	Since $\pi^1$ was arbitrary in $\Pi_M(\mu,\eta)$ we get that $R$ is affine with slope $1$ on $I$, whenever $I$ is an irreducible interval wrt $(\mu,\eta)$. Therefore $\eta$ and $\nu^*$ restricted to $I$ coincide. Hence, $\eta = \nu^*$.
	
	We finally show that $\nu^*$ is minimal in the convex order as stated. By Lemma~\ref{lem:minimum of maps}, we can assume $\mu\leq_c\eta\leq_c\nu^*$ and that $\eta$ can be pushed forward onto $\nu$ via an increasing $1$-Lipschitz map $S$. It follows by Lemma \ref{lem:maximum minimizes id-map} that $(id-S)(\eta)\leq_c (id-\tilde T)(\nu^*)$, so
	$$\textstyle V_\theta(\mu,\nu)\leq \textstyle\int \theta(x-S(x))\eta(dx)\leq \int \theta(x-\tilde T(x))\nu^*(dx)=V_\theta(\mu,\nu),$$
	and by the uniqueness obtained above we deduce $\eta=\nu^*$.
\end{proof}

\section{Stability of barycentric weak transport problems in multiple dimensions}

The final part of the article is concerned with stability of the weak optimal transport problem under barycentric costs, see \eqref{eq:GJ}. Unlike in the rest of the article we work here on $\R^d$. The final aim is to prove Theorem \ref{thm:stability'}.

One surprising aspect of this result is that we only require $\nu^k\rightarrow \nu$ in $\mathcal W_1$ and not necessarily in $\mathcal W_\rho$. This relates to the conditional expectation in \eqref{eq:GJ} being `inside of $\theta$.' 
%
We first prove an illuminating intermediate result:

\begin{proposition}\label{prop perturbation mu nu eta}
Let $1\leq \rho<\infty$, $(\mu^k)_{k\in\N}\in\mathcal P_\rho(\R^d)^\N$, $(\nu^k)_{k\in\N}\in\mathcal P_\rho(\R^d)^\N$ and $\theta\colon\R^d\rightarrow\R$ convex and satisfying the growth condition \eqref{eq:theta growth condition}. Suppose that $\mu^k\rightarrow \mu$ and $\nu^k\rightarrow \nu$ in $\mathcal W_\rho$, and that $\eta\leq_c \nu$. Then there exist $\eta^k\leq_c \nu^k$ such that
	\begin{itemize}
	\item[(i)] $\eta^k\rightarrow \eta$ in $\mathcal W_\rho$, 
	\item[(ii)] $\lim_k W_\theta(\mu^k,\eta^k) = W_\theta(\mu,\eta).$
	\end{itemize}
\end{proposition}

\begin{proof}
It is well-known that $(i)$ together with the stated convergence of the $\mu^k$'s implies $(ii)$, so we proceed to prove the former. Let $\pi^k$ be an optimal coupling attaining $W_\rho(\nu,\nu^k)$. Let $M$ be any martingale coupling with first marginal $\eta$ and second marginal $\nu$, the existence of which is guaranteed by the assumption $\eta\leq_c \nu$ together with Strassen's theorem. We convene on the notation $M(dx,dy)$ and $\pi(dy,dz)$,  and define the measure
$$P(dx,dy,dz)=M_y(dx)\pi^k_y(dz)\nu(dy).$$
This measure has $\eta$, $\nu$ and $\nu^k$ as first, second and third marginals. We next define $R^k(x)$ as the conditional expectation under $P$ of the third variable given the first one, namely
$$\textstyle R^k(x)=\int\int z\,\pi^k_y(dz)M_x(dy).$$
Next we introduce $\eta^k:=R^k(\eta)$ so by definition $\eta^k\leq_c\nu^k$. Finally
\begin{equation}\label{eq chain rho}
\begin{split}
\textstyle W_\rho(\eta,\eta^k)^\rho &\leq \textstyle  \int |x-R^k(x)|^\rho \eta(dx) = \int \left |\int y \, M_x(dy) - \int \left( \int z\,\pi^k_y(dz)\right) \, M_x(dy) \right |^\rho \eta(dx) \\
&\leq\textstyle  \int \left |y-\int z\,\pi^k_y(dz) \right |^\rho \nu(dy)  \leq  \int\int |y-z|^\rho \pi^k(dy,dz) \,=\, W_\rho(\nu,\nu^k),
\end{split}
\end{equation}
by the martingale property and two applications of Jensen's inequality. The desired conclusion follows.
\end{proof}

\begin{remark}\label{rem techni}
In the context of the previous proposition, if $\eta$ is supported in finitely many atoms, then the condition that $\nu^k\rightarrow \nu$ in $\mathcal W_\rho$ can be relaxed to convergence in $\mathcal W_1$. To wit, if $\eta = \sum_{i=1}^\ell \alpha_i\delta_{x^i}$, one can take $\rho=1$ in \eqref{eq chain rho} and prove 
$$\forall i\leq \ell:\,\,|x^i-R^k(x^i)|\leq (\min\{\alpha_j\})^{-1}W_1(\nu,\nu^k),$$
so taking $\rho$-power and integrating w.r.t.\ $\eta$ we get $W_\rho(\eta,\eta^k)^\rho\leq K\, W_1(\nu,\nu^k)^\rho\to 0$.
\end{remark}

The previous remark shows that we need to reduce to the finite-support setting. We carry to this in the next two lemmas:

\begin{lemma}\label{lem:convex order approximation mean}
Let $\eta \in \mathcal P_1(\R^d)$. Then for any $\epsilon>0$ there is a compactly supported, positive measures $\tilde \eta$ with
\begin{align}\label{eq:coapprox properties}\textstyle 
	\tilde \eta \leq \eta,\quad \tilde \eta(\R^d)\geq 1-\epsilon,\quad
	\int_{\R^d} z \eta(dz) = \frac{1}{\tilde \eta(\R^d)}\int_{\R^d} z\tilde\eta(dz).
\end{align}
\end{lemma}

\begin{proof}
	We first partition $\R^d$ into countable, disjoint $d$-dimensional cubes $(Q^\delta_k)_{k\in\N}$ of length $\delta > 0$. Define an approximation $\eta^\delta$ of $\eta$ by
	\begin{align*}\textstyle 
	\eta^\delta := \sum_{k\in\N} \delta_{z^\delta_k}\,\eta(Q^\delta_k),\quad z^\delta_k := \begin{cases}\frac{1}{\eta(Q^\delta_k)}\int_{Q^\delta_k} z \eta(dz)& \eta(Q^\delta_k) > 0, \\0 & \text{else}.\end{cases}
	\end{align*}
	Note that $\eta^\delta\leq_c \eta$ and $\eta^\delta \rightarrow \eta$ in $\mathcal W_1$ when $\delta\searrow 0$. If there exists an approximation $\eta^\delta$ such that the assertion holds, then it is straightforward to construct the corresponding measure for $\eta$, which in turn satisfies the assertion with respect to $\eta$. Wlog, we may assume that
	\begin{align*}\textstyle 
		\left\{\sum_{i=1}^{2d} \alpha_i v_i\colon (\alpha_i)_{i=1}^{2d} \in \R_+^{2d}, v_1,\ldots, v_{2d} \in \Big\{x-\bar z \in \R^d\colon x \in \supp(\eta)\Big\}\right\} = \R^d,
	\end{align*}
	where $\bar z$ denotes the barycenter of $\eta$. Then we can find $\delta>0$ such that
	\begin{align*}\textstyle 
		\left\{\sum_{i=1}^{2d} \alpha_i v_i\colon (\alpha_i)_{i=1}^{2d} \in \R_+^{2d}, v_1,\ldots, v_{2d} \in \Big\{x-\bar z \in \R^d\colon x \in \supp(\eta^\delta)\Big\}\right\} = \R^d.
	\end{align*}
	Let $z^\delta_{n_1},\ldots,z^\delta_{n_{2d}}$ span $\R^d$ in the sense above and 
	$$\eta^{\delta}(z^\delta_{n_j}) = \eta(Q^\delta_{n_j}) > 0\quad j=1,\ldots,2d.$$
	For any $\epsilon>0$ there is a $\tilde \epsilon \in (0,\epsilon)$ such that
	\begin{align*}\textstyle 
		\left\{\sum_{i=1}^{2d} \alpha_i z_{n_i}\colon (\alpha_i)_{i=1}^{2d}\in \R_+^{2d},\sum_{i=1}^{2d} \alpha_i < \epsilon\right\} \supset B_{\tilde \epsilon}(0).
	\end{align*}
	Besides, there exists a compact set $K\subset \R^d$ such that
	\begin{align*}\textstyle 
		\eta^\delta(K^c)<\tilde\epsilon,\quad \left|\bar z - \int_{K} z \eta^\delta(dz)\right| < \tilde \epsilon
	\end{align*}
	and $z^\delta_{n_1},\ldots,z^\delta_{n_{2d}} \in K$. Therefore, we find $(\tilde \alpha_i)_{i=1}^{2d} \in \R_+^{2d}$ with
	\begin{align*}\textstyle 
		\bar z - \int_K z\eta^{\delta}(dz) = \sum_{i=1}^{2d} \tilde \alpha_i z_{n_i},\quad \sum_{i=1}^{2d} \tilde \alpha_i < \epsilon.
	\end{align*}
	If $\epsilon$ is chosen smaller than $\eta^\delta(z_{n_i}^\delta)$ for all $i=1,\ldots,2d$, we can define the $\tilde \eta^\delta$ via
	\begin{align}{\textstyle {}\hspace{6cm}
		\tilde \eta^\delta := \eta^\delta{\upharpoonright_K} - \sum_{i=1}^{2d} \tilde \alpha_i \delta_{z_{n_i}}. \hspace{6cm}}\qedhere
	\end{align}
\end{proof}

\begin{lemma}\label{lem:Wtheta approximation}
	Let $\mu,\eta\in\mathcal P_\rho(\R^d)$ and $\theta\colon\R^d\rightarrow \R$ convex satisfying the growth condition \eqref{eq:theta growth condition}.
	Then there exists a sequence $(\eta^k)_{k\in\N}$ of finitely supported measures with
	$\eta^k\leq_c\eta$, $\eta^k\rightarrow\eta$ in $\mathcal W_\rho$ and $W_\theta(\mu,\eta^k)\rightarrow W_\theta(\mu,\eta)$.
\end{lemma}

\begin{proof}
	For any $\epsilon>0$ we find a compact set $K_\epsilon\subset \supp(\eta)$ such that
	$\int_{K_\epsilon^c}|y|^\rho\eta(dy) < \epsilon$.
	For any $\delta>0$, the set $K_\epsilon$ can be covered by finitely many, disjoint sets $(A^{\epsilon,\delta}_i)_{i=1}^{N_{\epsilon,\delta}}$ with diameter smaller than $\delta$ and $K_\epsilon = \bigcup_i A^{\epsilon,\delta}_i$. Define the measure $\eta^{\epsilon,\delta}$ by
\begin{align*}\textstyle 
	\eta^{\epsilon,\delta} = \delta_{z_\epsilon} \eta(K_\epsilon^c)  + \sum_{i=1}^{N_\delta} \delta_{z^{\epsilon,\delta}_i} \eta(A^{\epsilon,\delta}_i),
\end{align*}
where the points $z_\epsilon$ and $(z_i^{\epsilon,\delta})_{i=1}^{N_{\epsilon,\delta}}$ are given by
\begin{align*}\textstyle 
	z_\epsilon := \frac{1}{\eta(K_\epsilon^c)} \int_{K_\epsilon^c} y \eta(dy),\quad
	z_i^{\epsilon,\delta} := \frac{1}{\eta(A^{\epsilon,\delta}_i)}\int_{A^{\epsilon,\delta}_i} y \eta(dy).
\end{align*}
By construction, there exists a martingale coupling between $\eta^{\epsilon,\delta}$ and $\eta$, thus, $\eta^{\epsilon,\delta}\leq_c\eta$. 
 Note that $\theta$ restricted to $K_\epsilon$ is Lipschitz continuous. Then drawing the limit $\delta \rightarrow 0$ yields
\begin{align*}\textstyle 
	\sum_{i=1}^{N_{\epsilon,\delta}} \delta_{z^{\epsilon,\delta}_i} \eta(A_i^{\epsilon,\delta}) \rightarrow \eta|_{K_\epsilon}\text{ in }\mathcal W_\rho,
\end{align*}
and by convexity,
$\textstyle 
	|z_\epsilon|^\rho\eta(K_\epsilon^c) \leq \int_{K_\epsilon^c} |y|^\rho\eta(dy).$
Choosing $\delta(\epsilon)$ sufficiently small, we have $$\mathcal W_\rho(\eta,\eta^{\epsilon,\delta(\epsilon)}) \leq 2\epsilon\text{ and } \eta^{\epsilon,\delta(\epsilon)} \rightarrow \eta \text{ in }\mathcal W_\rho.$$
By the growth condition \eqref{eq:theta growth condition} and stability, we obtain $W_\theta(\mu,\eta^{\epsilon,\delta(\epsilon)})\to W_\theta(\mu,\eta)$.
\end{proof}

We can now prove a version of Proposition \ref{prop perturbation mu nu eta} under weaker assumptions:
\begin{lemma}\label{prop:Wtheta convex order convergence}
	Let $(\nu^k)_{k\in\N}$ be a sequence in $\mathcal P_1(\R^d)$ and let $(\mu^k)_{k\in\N}$ be a sequence $\mathcal P_\rho(\R^d)$  with $		\nu^k\rightarrow \nu $ in $\mathcal W_1$, $ \mu^k\rightarrow \mu $ in $\mathcal W_\rho$, 	where $\rho\geq 1$, and let $\theta\colon\R^d\rightarrow \R$ be a convex functions satisfying the growth constraint \eqref{eq:theta growth condition}. Then for any $\eta\leq_c\nu$ we find a sequence of $\eta^k\leq_c\nu^k$ such that $W_\theta(\mu^k,\eta^k) \rightarrow W_\theta(\mu,\eta)$ and $ \eta^k \rightarrow \eta $ in $ \mathcal W_1$.
\end{lemma}

\begin{proof}
	Wlog\ assume that $\theta$ is positive.
	By Lemma~\ref{lem:convex order approximation mean}, we can find for any $\epsilon>0$ a compactly supported $\hat \eta = \tilde \eta + (1-\tilde\eta(\R^d))\delta_{\bar z} \in \mathcal P_\rho(\R^d)$ with $\hat \eta \leq_c \eta$, $\mathcal W_1(\hat\eta,\eta) < \epsilon$ and
	\begin{align*}\textstyle 
		W_\theta(\mu,\hat \eta) \leq W_\theta(\mu,\eta) + c\int_{\R^d}1+|x-\bar z|^\rho \mu(dx) < \infty,
	\end{align*}
	where $\bar z := \int_{\R^d} y \eta(dy)$.
	Using stability of classical optimal transport, see \cite[Theorem 5.20]{Vi09}, we may assume that $|W_\theta(\mu,\hat \eta) - W_\theta(\mu,\eta)| < \epsilon$. 	By Lemma~\ref{lem:Wtheta approximation} we may reduce to the case of finitely supported $\hat \eta$. We conclude the proof with Remark \ref{rem techni}.
\end{proof}

Finally we can give the pending proof of Theorem \ref{thm:stability'}:

\begin{proof}[Proof of Theorem \ref{thm:stability'}]
	Lower-semicontinuity of the map $(\mu,\nu)\mapsto V_\theta(\mu,\nu)$ follows from \cite[Theorem 1.3]{BaBePa18}.  
	By \cite[Lemma 6.1]{BaBePa18} we have
	\begin{align}\label{eq:Vtheta=infWtheta}\textstyle 
		V_\theta(\mu,\nu) = \inf_{\eta\leq_c\nu} W_\theta(\mu,\eta),
	\end{align}
	where the infimum is even attained for a measure $\eta\leq_c\nu$. By Lemma~\ref{prop:Wtheta convex order convergence} we find a sequence $\eta^k\leq_c\nu^k$, so that again using \cite[Lemma 6.1]{BaBePa18} we find
	\begin{align*}\textstyle 
		V_\theta(\mu,\nu)=W_\theta(\mu,\eta)= \lim_k W_\theta(\mu^k,\eta^k) \geq \limsup_k V_\theta(\mu^k,\nu^k).
	\end{align*}
	If $\theta$ is strictly convex, the infimum in \eqref{eq:Vtheta=infWtheta} is attained by a unique probability measure $\eta^k\leq_c\nu^k$,\footnote{The uniqueness of $\eta^k$, $T^k$ and $\pi^k$ was already shown in \cite[Theorem 2.1]{AlCoJo17} for $|\cdot|^\rho$, $\rho > 1$} which  in turn is the push-forward of $\mu^k$ under a $\mu^k$-uniquely defined map $T^k$. Moreover, the $W_\theta$-optimal transport plan $\pi^k \in \Pi(\mu^k,\nu^k)$ is uniquely determined by $\mu^k(dx)\delta_{T^k(x)}(dy)$: Suppose the contrary and let $T'(x) := \int_{\R^d} y \pi^k_x(dy)$, then $T'(\mu^k) \leq_c \nu$ and we find the contradiction
	\begin{align*}\textstyle 
		V_\theta(\mu^k,\nu^k) \leq \int_{\R^d} \theta(x-T'(x)) \mu(dx) < \int_{\R^d\times \R^d} \theta(x - y) \pi^k(dx,dy) = V_\theta(\mu^k,\nu^k).
	\end{align*}
	Hence, by convergence of the values of $V_\theta$ and tightness of $(\eta_k)_{k\in\N}$, we deduce the convergence of the $\eta^k$ to the optimal $\eta \leq_c \nu$ in $\mathcal W_1$. Suppose that $\mu^k = \mu$ for all $k\in\N$, then due to the uniqueness of the optimal transport maps $T$ between $T$ and $T(\mu$), we can apply Theorem~\cite[Corollary 5.23]{Vi09} and obtain convergence of the transport maps $T^k$ to $T$.
\end{proof}

\bibliography{joint_biblio}{}

\begin{thebibliography}{10}

\bibitem{AcBaZa16}
B.~{Acciaio}, J.~{Backhoff-Veraguas}, and A.~{Zalashko}.
\newblock {Causal optimal transport and its links to enlargement of filtrations
  and continuous-time stochastic optimization}.
\newblock {\em ArXiv e-prints}, 2016.

\bibitem{AlCoJo17}
A.~{Alfonsi}, J.~{Corbetta}, and B.~{Jourdain}.
\newblock {Sampling of probability measures in the convex order and
  approximation of Martingale Optimal Transport problems}.
\newblock {\em ArXiv e-prints}, Sept. 2017.

\bibitem{AlCoJo19}
A.~{Alfonsi}, J.~{Corbetta}, and B.~{Jourdain}.
\newblock {Sampling of probability measures in the convex order by Wasserstein
  projection}.
\newblock {\em arXiv e-prints}, page arXiv:1709.05287, Feb. 2019.

\bibitem{AlBoCh18}
J.-J. Alibert, G.~Bouchitte, and T.~Champion.
\newblock A new class of cost for optimal transport planning.
\newblock {\em hal-preprint}, 2018.

\bibitem{BaBaBeEd19}
J.~Backhoff-Veraguas, D.~Bartl, M.~Beiglb\"ock, and M.~Eder.
\newblock Adapted wasserstein distances and stability in mathematical finance.
\newblock {\em ArXiv e-prints}, 2019.

\bibitem{BaBeHuKa17}
J.~{Backhoff-Veraguas}, M.~{Beiglb{\"o}ck}, M.~{Huesmann}, and
  S.~{K{\"a}llblad}.
\newblock {Martingale Benamou--Brenier: a probabilistic perspective}.
\newblock {\em ArXiv e-prints}, Aug. 2017.

\bibitem{BaBeLiZa16}
J.~Backhoff-Veraguas, M.~Beiglb{\"o}ck, Y.~Lin, and A.~Zalashko.
\newblock Causal transport in discrete time and applications.
\newblock {\em SIAM Journal on Optimization}, 27(4):2528--2562, 2017.

\bibitem{BaBePa18}
J.~Backhoff-Veraguas, M.~Beiglb\"ock, and G.~Pammer.
\newblock Existence, duality, and cyclical monotonicity for weak transport
  costs.
\newblock {\em ArXiv e-prints}, 2018.

\bibitem{BeJu16}
M.~Beiglb{\"o}ck and N.~Juillet.
\newblock On a problem of optimal transport under marginal martingale
  constraints.
\newblock {\em Ann. Probab.}, 44(1):42--106, 2016.

\bibitem{BeJu17}
M.~{Beiglb\"ock} and N.~{Juillet}.
\newblock {Shadow couplings}.
\newblock {\em ArXiv e-prints}, Sept. 2016.

\bibitem{FaSh18}
M.~Fathi and Y.~Shu.
\newblock Curvature and transport inequalities for {M}arkov chains in discrete
  spaces.
\newblock {\em Bernoulli}, 24(1):672--698, 2018.

\bibitem{GaMc96}
W.~Gangbo and R.~McCann.
\newblock The geometry of optimal transportation.
\newblock {\em Acta Math.}, 177(2):113--161, 1996.

\bibitem{GoJu18}
N.~Gozlan and N.~Juillet.
\newblock On a mixture of brenier and strassen theorems.
\newblock {\em arXiv preprint arXiv:1808.02681}, 2018.

\bibitem{GoRoSaSh18}
N.~Gozlan, C.~Roberto, P.-M. Samson, Y.~Shu, and P.~Tetali.
\newblock Characterization of a class of weak transport-entropy inequalities on
  the line.
\newblock {\em Ann. Inst. Henri Poincar\'e Probab. Stat.}, 54(3):1667--1693,
  2018.

\bibitem{GoRoSaTe17}
N.~Gozlan, C.~Roberto, P.-M. Samson, and P.~Tetali.
\newblock Kantorovich duality for general transport costs and applications.
\newblock {\em J. Funct. Anal.}, 273(11):3327--3405, 2017.

\bibitem{Ma96contracting}
K.~Marton.
\newblock A measure concentration inequality for contracting markov chains.
\newblock {\em Geometric \& Functional Analysis GAFA}, 6(3):556--571, 1996.

\bibitem{Ma96concentration}
K.~Marton et~al.
\newblock Bounding $\bar d$-distance by informational divergence: A method to
  prove measure concentration.
\newblock {\em The Annals of Probability}, 24(2):857--866, 1996.

\bibitem{Sa17}
P.-M. Samson.
\newblock Transport-entropy inequalities on locally acting groups of
  permutations.
\newblock {\em Electron. J. Probab.}, 22:Paper No. 62, 33, 2017.

\bibitem{Sh16}
Y.~Shu.
\newblock From hopf-lax formula to optimal weak transfer plan.
\newblock {\em arXiv preprint arXiv:1609.03405}, 2016.

\bibitem{Sh18}
Y.~Shu.
\newblock Hamilton-{J}acobi equations on graph and applications.
\newblock {\em Potential Anal.}, 48(2):125--157, 2018.

\bibitem{Ta95}
M.~Talagrand.
\newblock Concentration of measure and isoperimetric inequalities in product
  spaces.
\newblock {\em Publications Math{\'e}matiques de l'Institut des Hautes Etudes
  Scientifiques}, 81(1):73--205, 1995.

\bibitem{Ta96}
M.~Talagrand.
\newblock New concentration inequalities in product spaces.
\newblock {\em Inventiones mathematicae}, 126(3):505--563, 1996.

\bibitem{Vi09}
C.~Villani.
\newblock {\em Optimal Transport. Old and New}, volume 338 of {\em Grundlehren
  der mathematischen Wissenschaften}.
\newblock Springer, 2009.

\end{thebibliography}
\bibliographystyle{abbrv}
\end{document}